\documentclass[12pt]{amsart}

\usepackage{latexsym,amsfonts,amsmath,amssymb,epsfig,tabularx,amsthm,dsfont}

\usepackage{accents}

\theoremstyle{definition}

\newtheorem{thm}{Theorem} 
\newtheorem{rem}[thm]{Remark}
\newtheorem{lem}[thm]{Lemma}
\newtheorem{prop}[thm]{Proposition}

\newcommand{\Z}{\mathbb{Z}}
\newcommand{\N}{\mathbb{N}}
\newcommand{\R}{\mathbb{R}}
\newcommand{\ind}{\scalebox{1.2}{\raisebox{-0.2mm}{$\mathds{1}$}}}
\renewcommand{\l}{\langle}
\renewcommand{\r}{\rangle}
\renewcommand{\a}{\alpha}

\renewcommand{\O}{\Omega}
\newcommand{\wt}{\widetilde}
\newcommand{\widebar}[1]{\mbox{\kern1.5pt\hbox{\vbox{\hrule height 0.6pt \kern0.35ex
        \hbox{\kern-0.15em \ensuremath{#1 }\kern0.0em}}}}\kern-0.1pt}
\newcommand\dint{{\,\rm d}}
\newcommand{\vol}{{\rm vol}}
\newcommand{\supp}{{\rm supp}}
\newcommand{\diag}{{\rm diag}}

\newcommand{\mr}[1]{\accentset{\circ}{#1}}

\newcommand{\Pn}{\mathcal{P}_n}
\newcommand{\F}{\mathcal{F}}
\newcommand{\E}{\mathbb{E}}
\newcommand{\D}{\Delta}

\newcommand{\U}{\mathcal{U}}

\title[Monte Carlo for smooth functions]
{A Monte Carlo method for integration of multivariate smooth functions}   

\author{Mario Ullrich}   
\address{Institut f\"ur Analysis, Johannes Kepler Universit\"at Linz, Austria}
\email{mario.ullrich@jku.at}

\date{\today}



\begin{document} 
 
\begin{abstract} 
We study a Monte Carlo algorithm that is based on a specific (randomly shifted and dilated) 
lattice point set. 
The main result of this paper is that the mean squared error for a given compactly supported, 
square-integrable 
function is bounded by $n^{-1/2}$ times the $L_2$-norm of the Fourier transform outside a 
region around the origin, where $n$ is the expected number of function evaluations. 
As corollaries we obtain the optimal order of convergence for functions from the Sobolev 
spaces $H^s_p$ with isotropic, anisotropic or mixed smoothness with given 
compact support for all values of the parameters. 
If the region of integration is the unit cube, we obtain the same optimal orders 
for functions without boundary conditions.
This proves, in particular, that the optimal order of convergence in the latter 
case is $n^{-s-1/2}$ for $p\ge2$, which is, in contrast to the case of 
deterministic algorithms, independent of the dimension. 
This shows that Monte Carlo algorithms can improve the order by more than $n^{-1/2}$ 
for a whole class of natural function spaces.
Note that a similar result (for a different class) was obtained by 
Heinrich et al.~\cite{HHY04}.
\end{abstract} 

\maketitle 

\section{Introduction}

We study Monte Carlo methods, i.e.~randomized cubature formulas, 
for the approximation of the $d$-dimensional integral
\[
I(f) \,=\, \int_{\O} f(x) \dint x, 
\]
where $\O\subset\R^d$ is a bounded, measurable set with an interior point
and $f\colon\R^d\to\R$ is an 
integrable function with support inside $\O$, i.e. 
$\supp(f):=\{x\in\R^d\colon f(x)\neq 0\}\subset\Omega$.
Without loss of generality we assume that $\Omega$ has volume 1. 
In the case $\Omega=[0,1]^d$ we will also study functions 
without boundary conditions, i.e. without the restriction that the support 
is contained in $[0,1]^d$, see Section~\ref{sec:non-periodic}.

The randomized algorithms under consideration are of the form
\begin{equation} \label{eq:algorithms}
M(f) \,=\, \sum_{j=1}^m a_j f(x^j), 
\end{equation}
where the nodes $x^j=(x^j_1,\dots,x^j_d)\in\O$, 
the weigths $a_j\in\R$, $j=1,\dots,m$, 
and the number of points $m\in\N$ are random variables. 
Let $N(M):=\E[m]$ be the expected number of function evaluations 
that are used by $M$.

The algorithm we want to study was introduced recently by Krieg and 
Novak~\cite{KN16} and is based on the deterministic cubature rule of 
Frolov~\cite{Fr76}, which attracted some attention in the past years 
due to its optimality (in order) for numerical integration in nearly 
every classical function space on the cube, see 
e.g.~\cite{DP14,Du93,Du97,DTU16,No16,Sk94,Te93,Te03,MU15} or \cite{UU16} for a 
recent survey of known results. 
We are not aware of an example of a natural function space on the cube, where Frolov's 
cubature rule, combined with some modification for non-periodic functions, 
see Section~\ref{sec:non-periodic}, 
is demonstrable not optimal.

Here we continue the analysis from \cite{KN16} and analyze the following 
random algorithm: \\
Let $B_n\in\R^{d\times d}$, $n>0$, be a suitable sequence of invertible matrices, 
i.e.~we need that the $B_n$ satisfy $\det(B_n)=n$ and~\eqref{eq:prop}.
Let 
$u=(u_1,\dots,u_d)\sim\U([1/2,3/2]^d)$ and 
$v=(v_1,\dots,v_d)\sim\U([0,1]^d)$ be two uniformly distributed random vectors. 
We consider the Monte Carlo method
\begin{equation}\label{eq:alg}
M_n(f) \,:=\, \frac{1}{n}\sum_{x\in\Pn} f(x),
\end{equation}
where
\begin{equation} \label{eq:points}
\begin{split}
\Pn \,&:=\, \Omega \,\cap\, (U B_n)^{-\top}(\Z^d + v) \\
\,&\;=\, \Omega \,\cap\,\left\{B_n^{-\top}(z)\colon 
z=\left(\frac{m_1+v_1}{u_1},\ldots,\frac{m_d+v_d}{u_d}\right),\, 
m\in\Z^d\right\}, 
\end{split}
\end{equation}
$B_n^{-\top}$ is the transposed inverse of $B_n$ and $U=\diag(u)$.
Note that this method has equal weights and satisfies 
$N(M_n)=n$, see~\eqref{eq:number}.


Define the \emph{root mean square error} of a randomized algorithm $M$ 
for a specific function $f\in L_1(\R^d)$ by
\[
\D(M, f) \,:=\, \left(\E\bigl[|I(f) - M(f)|^2\bigr]\right)^{1/2}.
\]
and let 
\begin{equation}\label{eq:Lpo}
L_p^\circ(\Omega) \,:=\, \{f\in L_p(\R^d)\colon \supp(f)\subset\Omega\}.
\end{equation}
\medskip

We will prove the following theorem.

\begin{thm}\label{thm:error-L2}
Let $M_n$ be given by \eqref{eq:alg} and $f\in L_2^\circ(\Omega)$. 
Then,
\[
\Delta(M_n,f) \,\lesssim\, n^{-1/2}\, \|\F f\|_{L_2(D_n)},
\]
where $D_n=\{\xi\in\R^d \colon \prod_{j=1}^d|\xi_j|\gtrsim n\}$
and $\F f$ is the Fourier transform of $f$.
\end{thm}

The proof of Theorem~\ref{thm:error-L2} without hidden constants is given in
Section~\ref{subsec:dilation}.
\medskip

We apply Theorem~\ref{thm:error-L2} to obtain error bounds for Sobolev 
spaces with isotropic and mixed smoothness. 
Here we only comment on the results for Sobolev spaces with integer smoothness. 
For the general statement of the results, also in the anisotropic setting, 
see Section~\ref{sec:sobolev}.

In detail, for $s\in\N$ and $1\le p\le\infty$, we consider the \emph{isotropic 
Sobolev spaces} 
\begin{equation*} 
\mr{H}_p^s(\O) \,:=\, \left\{ f\in L_p^\circ(\O)\colon D^\a f\in L_p(\R^d) 
	\text{ for } |\a|_1\le s \right\}
\end{equation*}
and the \emph{mixed Sobolev spaces}
\begin{equation*} 
\mr{\bf H}_p^s(\O) \,:=\, \left\{ f\in L_p^\circ(\O)\colon D^\a f\in L_p(\R^d) 
	\text{ for } |\a|_\infty\le s\right\}
\end{equation*}
equipped with the norms
\begin{equation*}
\|f\|_{H^s_p} \,=\, \|f\|_{L_p}+\sum_{j=1}^d \|D^{s\cdot e_j} f\|_{L_p}
\end{equation*}
and
\begin{equation*}
\|f\|_{{\bf H}^s_p} \,=\, \|f\|_{L_p}+ \sum_{\a\colon \alpha_j\in\{0, s\}}\|D^{\a} f\|_{L_p},
\end{equation*} 
respectively, where $D^\a f$, $\a\in\N^d_0$, denotes the usual weak 
partial derivative of a function~$f$ and 
$e_j$ is the $j$th unit vector in $\R^d$. 
Recall from~\eqref{eq:Lpo} that functions from $\mr{H}_p^s(\O)$ and 
$\mr{\bf H}_p^s(\O)$, respectively, have support inside the bounded, 
measurable set $\O\subset\R^d$.

Let
\[
\sigma_p \,:=\, 
\max\left\{0, \frac1p-\frac12\right\}.
\]
We prove that, for $1\le p\le\infty$,
\[
\Delta(M_n, f) \,\lesssim\, 
n^{-s/d-1/2 \,+\, \sigma_p}\, \|f\|_{H^s_p}
\]
for $f\in\mr{H}^s_p(\O)$ with $s/d\ge\sigma_p$, 
see Theorem~\ref{thm:main1},
and 
\[
\Delta(M_n, f) \,\lesssim\, 
n^{-s-1/2 \,+\, \sigma_p}\, \|f\|_{{\bf H}^s_p} \\
\]
for $f\in\mr{\bf H}^s_p(\O)$ with $s\ge\sigma_p$,
see Theorem~\ref{thm:main2}.
Note that for $p\ge2$ and $s\ge0$ the result for mixed Sobolev spaces reads
\[
\Delta(M_n, f) \,\lesssim\, n^{-s-1/2}\; \|f\|_{{\bf H}^s_p}.
\]
In Section~\ref{sec:non-periodic} we present a modification of the algorithm 
that has the same orders of convergence for functions defined on the unit 
cube $[0,1]^d$ without boundary conditions.

\medskip

For other algorithms the upper bound for isotropic spaces 
is known for 
some time and this order of $n$ cannot be improved by any other algorithm, 
see e.g.~Heinrich~\cite{He93} or Novak~\cite{No88}. 
The algorithms are based on 
($L_p$-)approximation of the integrand and 
the standard Monte Carlo method applied to the residual. 
This works since 
the optimal order for approximation and integration is the same for isotropic spaces. 
However, this method is not quite practical. 
For mixed Sobolev spaces the optimal order for approximation is different, 
see e.g.~the survey~\cite{DTU16}, 
and hence, this technique does not lead to an optimal result. 
For other approaches to randomized numerical integration and for 
results for other function spaces see 
e.g.~\cite{Ba59, Ba62, Ba15, HHY04, H10, KWW06, NW10, SKJ02}.

The case of deterministic algorithms is better understood, see~\cite{DP14,DTU16,HNUW16,No16,Te93,Te03}. 
E.g., it is known that the optimal order for deterministic algorithms 
in $H^s_p([0,1]^d)$ and ${\bf H}^s_p([0,1]^d)$, see Section~\ref{sec:non-periodic}, 
is $n^{-s/d}$ for $s/d>1/p$, and $n^{-s}(\log n)^{(d-1)/2}$ for $s>\max\{1/p,1/2\}$, 
respectively. The restriction to $s/d>1/p$ (resp.~$s>1/p$) is necessary to ensure 
that the functions are continuous. 
In particular, these optimal orders are achieved by Frolov's cubature rule, 
which is the deterministic cubature rule given by \eqref{eq:alg} and \eqref{eq:points} 
with the random elements $u$ and $v$ replaced by $(1,\dots,1)$ and $(0,\dots,0)$, 
respectively,
see e.g.~\cite{Te03}. 
For $p>2$ and $1/p<s<1/2$ the optimal order for ${\bf H}^s_p([0,1]^d)$ is still not known, 
even for $d=2$. See \cite{UU16} for some recent progress on the upper bound in this range.

The randomized algorithm $M_n$ from \eqref{eq:alg} was first considered in \cite{KN16}. 
The idea behind the algorithm is similar to the one of 
Bakhvalov~\cite{Ba62}, who analyzed an integration lattice 
rule (of Korobov type) with a random number of points. 
In~\cite{KN16} the optimal order of $M_n$ for the isotropic Sobolev spaces $H^s_2$ with 
$s\in\N$ and $s/d>1/2$ is proven. 
The authors also show the (not optimal) upper bound $n^{-s-1/2}\, (\log n)^{(d-1)/2}$ 
for ${\bf H}_2^s$ with $s\in\N$. 

Here, we generalize the results of~\cite{KN16} to $p\neq2$, $s\notin\N$ and 
to anisotropic smoothness. We also consider the case of discontinuous functions, 
i.e.~$0\le s/d\le1/p$ and $0\le s\le1/p$ for isotropic and mixed Sobolev spaces, respectively. 
Moreover, we improve the upper bound by a 
certain power of $\log n$, i.e., we show that there 
is no logarithm at all in the upper bound. 
This bound is optimal. 
For this note that, by the results of~\cite{NUU16}, integration in the space 
${\bf H}_p^s([0,1]^d)$ 
is not harder than integration in 
$\mr{\bf H}_p^s$ with $\Omega=[0,1]^d$. 
Moreover, it is obvious that lower bounds for the 
one-dimensional classes ${\bf H}_p^s([0,1])=H_p^s([0,1])$ also hold for 
${\bf H}_p^s([0,1]^d)$ and the optimal order for these classes is $n^{-s-1/2+\sigma_p}$, see 
e.g.~\cite{He93, No88}. 
The optimality in order for general $\O$ then follows from the existence of a 
(possibly very small) cube inside $\Omega$. 
Hence, we obtain the following theorem on the optimal order for the worst case error 
of randomized algorithms for mixed Sobolev spaces.
For a normed space of functions $F$, let
\[
\Delta(M,F) \,:=\, \sup_{f\in F}\,\frac{\Delta(M,f)}{\|f\|_F}.
\]

\begin{thm}\label{thm:optimal} 
Let $s\ge0$ and $1\le p\le\infty$ ($1<p<\infty$ if $s\notin\N$)
with $s\ge\sigma_p$
and $\O$ be a bounded, measurable set 
with an interior point. 
We have 
\[
\inf_{M}\, \Delta\bigl(M,\mr{\bf H}^s_p(\O)\bigr) \,\asymp\, n^{-s-1/2+\sigma_p}
\]
and, for $\O=[0,1]^d$,
\[
\inf_{M}\, \Delta\bigl(M,{\bf H}^s_p([0,1]^d)\bigr) 
\,\asymp\, \inf_{M}\, \Delta\bigl(M,\mr{\bf H}^s_p([0,1]^d)\bigr) \,\asymp\, n^{-s-1/2+\sigma_p},
\]
where the infima are taken over all algorithms of the form~\eqref{eq:algorithms} 
with $N(M)\le n$.
\end{thm}

It is interesting to note that the optimal order for isotropic Sobolev spaces 
$H^s_p([0,1]^d)$ immediately follows from Theorem~\ref{thm:optimal} and 
the embedding $H^s_p\hookrightarrow{\bf H}^{s/d}_p$.

\bigskip
{\bf Notation.} As usual $\N$ denotes the natural numbers, $\N_0=\N\cup\{0\}$, 
$\Z$ denotes the integers and  
$\R$ (resp.~$\R_+$) the real (resp.~nonnegative) numbers. 
The letter $d$ is always reserved for the underlying dimension in $\R^d, \Z^d$ etc. We denote
by $\langle x,y\rangle$ or $x y$ the usual Euclidean inner product in $\R^d$.
For $a\in\R$ let $\lfloor a\rfloor\in\Z$ be the largest integer smaller or 
equal to $a$.
For $0<p\leq \infty$ and $x\in \R^d$ we let $|x|_p = (\sum_{i=1}^d |x_i|^p)^{1/p}$ 
with the usual modification in the case $p=\infty$. 
We further denote by $L_p(\R^d)$ the space of Lebesgue-measurable functions 
$f\colon\R^d\to\R$ such that $\|f\|_p:=(\int_{\R^d}|f(x)|^p\dint x)^{1/p}<\infty$. 
By $x\le y$ for $x,y\in\R^d$ we mean that the inequality holds component-wise.
For $u=(u_1,\dots,u_d)\in\R^d$ we write $\diag(u)$ for the $d\times d$-diagonal 
matrix with diagonal entries $u_1,\dots,u_d$.
For a bounded set $A\subset\R^d$ with positive volume we write 
$\U(A)$ for the uniform distribution in $A$. 
The logarithm $\log$ will always be in base 2. 
If $X$ and $Y$ are two (quasi-)normed spaces, the (quasi-)norm
of an element $x$ in $X$ will be denoted by $\|x\|_X$. 
The symbol $X \hookrightarrow Y$ indicates that the identity operator is continuous. 
For two sequences of real numbers $a_n$ and $b_n$ we will write 
$a_n \lesssim b_n$ if there exists a constant $c>0$ such that 
$a_n \leq c\,b_n$ for all $n$. We will write $a_n \asymp b_n$ if 
$a_n \lesssim b_n$ and $b_n \lesssim a_n$.

\medskip
\section{Preliminaries}

In this section we provide the tools that are needed to prove our results.
That is, we give a detailed description of the algorithm under consideration 
together with the important properties of the underlying deterministic 
point set and state Poisson's summation formula.

\subsection{The algorithm}\label{subsec:alg}

We analyze the algorithm that was introduced by Krieg and Novak~\cite{KN16} and which 
is based on the cubature rule of Frolov~\cite{Fr76}.

For this, consider an invertible matrix $B\in\R^{d\times d}$ and define the cubature rule

\begin{equation}\label{eq:alg_det}
Q_{B,v}(f) \,=\, \frac{1}{|\det B|}\sum_{m\in\Z^d} f\left(B^{-\top}(m+v)\right)
\end{equation}
where $v\in[0,1]^d$.
We follow~\cite{Fr76} and choose a (generator) matrix $B\in\R^{d\times d}$ 
with the property
\begin{equation}\label{eq:prop1}
\prod_{j=1}^d |(Bm)_j| \,\ge\,1 \quad\text{ for all } m\in\Z^d\setminus\{0\}.
\end{equation}
We will call such a matrix $B$ a \emph{Frolov matrix}. 
Clearly, every Frolov matrix is invertible.
For constructions of such matrices $B$ see e.g.~\cite{Fr76,Te93,MU15}.

\begin{rem}
It is proven in \cite[Lemma~3.1]{Sk94} that the property \eqref{eq:prop1} 
for $B$ is equivalent to the same property for $c B^{-\top}$ with some 
$c<\infty$. In numerical experiments one could therefore interchange the 
roles of $B$ and $B^{-\top}$ 
and use the lattice points $B(\Z^d+v)$ in \eqref{eq:alg_det}. 
We use this definition to ease the notation.
\end{rem}

Let $d_B:=\det(B)$ and define, for $n\in\R$, 
the matrices $B_n:=(n/d_B)^{1/d}B$. 
These matrices clearly satisfy $\det(B_n)=n$ and 
\begin{equation}\label{eq:prop}
\prod_{j=1}^d |(B_n m)_j| \,\ge\,n/d_B \quad\text{ for all } m\in\Z^d\setminus\{0\}.
\end{equation}


The \emph{randomized Frolov cubature rule} $M_n$ uses the two independent random vectors 
$u$ and $v$ that are uniformly distributed in $[1/2,3/2]^d$ and $[0,1]^d$, respectively.
We define the $d\times d$-diagonal matrix $U={\rm diag}(u)$.
Then, in view of~\eqref{eq:alg} and~\eqref{eq:alg_det} we have 
\begin{equation*}
M_n(f) \,=\, Q_{U B_n,v}(f).
\end{equation*}
We call $u$ (resp. $U$) the 
\emph{random dilation} and $v$ the \emph{random shift} of the algorithm~$M_n$.

It is known from \cite[Lemma~3]{KN16} that $M_n$ is well-defined and unbiased on $L_1(\R^d)$.
Moreover, if we consider functions that are supported in a bounded, measurable 
set $\Omega\subset\R^d$ with $\vol_d(\O)=1$, we know that 
the expected number of (non-zero) function evaluations that are used by the algorithm 
$M_n$, i.e. $N(M_n)$, equals $n$.
To see this, note that 
\begin{equation}\label{eq:number}
\begin{split}
N(M_n) \,&=\, \E\left[\sum_{m\in\Z^d} \ind_\O\left((UB_n)^{-\top}(m+v)\right)\right]
\,=\, \E\left[\sum_{m\in\Z^d} \ind\!\Bigl(m+v\in(UB_n)^{\top}(\O)\Bigr)\right]\\
\,&=\, \E_u\left[\sum_{m\in\Z^d} \vol_d\left((m+[0,1]^d)\cap(UB_n)^{\top}(\O)\right)\right]
\,=\, \E_u\left[\vol_d\left((UB_n)^\top(\O)\right)\right]\\
\,&=\, \E_u\left[\det(UB_n)\right] \, \vol_d(\O) 
\,=\, n\cdot\vol_d(\O) 
\,=\, n.
\end{split}
\end{equation}

\begin{rem}
The choice of the set $[1/2,3/2]^d$ for the random dilataion is quite arbitrary. 
Every set of the form $[1-c,1+c]^d$ with $c\in(0,1)$ would lead to the same results. 
However, the choice $c=1/2$ optimizes the constant in our upper bound.
\end{rem}


\subsection{Counting lattice points in boxes}

We still have to exploit the crucial property of the 
Frolov matrices that are used to construct our cubature rule. 
This property is, besides the fact that 
$B_n^{-\top}(\Z^d)$ is a lattice, 
that one can easily bound the number of points of the \emph{dual lattice} 
$B_n(\Z^d)$ in axis-parallel boxes.

There are many references that study this problem and state the following 
bound together with further properties of such lattices, see 
e.g.~\cite{Fr76,Fr77,Fr80,Le69,Sk94,Te93,MU15}. However, we only need 
a special case here and we give the short proof for convenience.

\begin{lem} \label{lem:boxes}
Let $B_n$ satisfy \eqref{eq:prop}. 
Then, for each axis-parallel box $R\subset\R^d$ containing the 
origin we have 
\[
\Bigl\vert B_n\bigl(\Z^d\setminus\{0\}\bigr)\cap R \Bigr\vert 
\,\le\, d_B\, \frac{\vol_d(R)}n.
\]
In particular, the left hand side is zero if $\vol_d(R)<n/d_B$.
\end{lem}

\begin{proof}
From \eqref{eq:prop}, together with the fact that $B_n(\Z^d)$ is a lattice, 
we obtain that every axis-parallel box $R'$ that contains at least two 
points $x,y\in B_n(\Z^d)$ must satisfy 
$\vol_d(R')\ge\prod_{j=1}^d|x_j-y_j|\ge n/d_B$.
Here we used that $x-y\in B_n\bigl(\Z^d\setminus\{0\}\bigr)$. 
Now we divide the box $R$ into $\lfloor d_B\cdot\vol_d(R)/n + 1\rfloor$ 
axis-parallel boxes of volume smaller $n/d_B$, which consequently contain 
at most one point. Moreover, by assumption, one of these boxes is empty. 
This proves the upper bound 
$\lfloor d_B\cdot\vol_d(R)/n + 1\rfloor-1 \le d_B\cdot\vol_d(R)/n$. \\
\end{proof}


For a comment on the magnitude of the constant $d_B$ see Remark~\ref{rem:dB}.


\subsection{Poisson's summation formula}

The Fourier transform of a function $f\in L_1(\R^d)$ is defined by 
\[
\F f(\xi) \,=\, \int_{\R^d} f(x)\, e^{-2\pi i \l \xi,x\r} \dint x, \qquad \xi\in\R^d,
\]
and the inverse Fourier transform is given by $\F^{-1}f(\xi)=\F f(-\xi)$.

The analysis of the error of cubature formulas that use nodes from a lattice 
is naturally related to 
an application of Poisson's summation formula and variations thereof.
A more detailed treatment and a proof of the following lemma 
can be found, e.g., in~\cite[Thm.~VII.2.4 \& Cor.~VII.2.6]{SW71}.

\begin{lem}\label{lem:periodization}
Let $f\in L_2^\circ(\O')$ for some bounded $\O'\subset\R^d$. 
Then its periodization 
$\sum_{\ell\in\Z^d} f(\ell+x)$ is a (1-periodic) function in 
$L_2([0,1]^d)$ that has the 
Fourier expansion
\[
\sum_{k\in\Z^d} \F f(k)\, e^{2\pi i \l k, x\r}.
\]
\end{lem}
\medskip


\section{The general error bound} \label{sec:error-main}

We now prove the most general form of our main result. 
We will do this in two sections to treat the random shift and the 
random dilation separately.

\subsection{Random shift}
The following lemma improves on 
\cite[Lemma~2]{KN16} and is one of the key 
ingredients in our proof.

\medskip
\begin{lem}\label{lem:bound-L2}
Let $B\in\R^{d\times d}$ be an invertible matrix, $f\in L_2^\circ(\O)$ 
and $v\sim\U([0,1]^d)$. Then, 
\[
\E_v\left[|I(f)-Q_{B,v}(f)|^2\right] \;=\; 
\sum_{k\in\Z^d\setminus\{0\}} |\F f(Bk)|^2.
\]
\end{lem}
\medskip

%

\begin{proof}
If we consider $Q_{B,v}(f)$, see \eqref{eq:alg_det}, as a function of 
$v\in[0,1]^d$ we easily obtain from 
Lemma~\ref{lem:periodization} that 
\[
Q_{B,v}(f) \,=\, \sum_{k\in\Z^d} \F f(Bk)\, e^{2\pi i \l k, v\r}
\]
for almost every $v\in[0,1]^d$.
Just apply Lemma~\ref{lem:periodization} to $g(x)=f(B^{-\top}x)$ and use 
that $\F g(k) = |\det(B)|\, \F f(Bk)$, 
which is possible since $g\in L_2^\circ(\O')$ with $\O'=B^\top(\O)$ if 
$f\in L_2^\circ(\O)$.
This also shows that
$Q_{B,v}(f)$ is a function (in $v$) that belongs to $L_2([0,1]^d)$.
Since $I(f)=\F f(0)$ and the desired expectation is nothing but the squared 
$L_2([0,1]^d)$-norm of this Fourier series, 
the results follows from Parseval's identity.\\
\end{proof}

\medskip


\subsection{Random dilation} \label{subsec:dilation}

We now show how the random dilation of the point set, see~\eqref{eq:points}, 
leads to our main error bound, i.e.~a bound on the root mean square error 
of $M_n(f)$ in terms of a certain $L_2$-norm of the Fourier transform of $f$. 
This proves Theorem~\ref{thm:error-L2}.
The proof is quite similar to the one in~\cite{KN16}.

\medskip
\noindent{\bf Theorem 1'.}
Let $M_n$, $n>0$, be given by \eqref{eq:alg} and $f\in L_2^\circ(\O)$. 
Moreover, we define the set 
$D_n=\{\xi\in\R^d \colon \prod_{j=1}^d|2 \xi_j|\ge n/d_B\}$. 
Then,
\[
\Delta(M_n,f) \,\le\, c_d\, n^{-1/2}\, \|\F f\|_{L_2(D_n)}
\]
with $c_d=3^{d/2}\sqrt{d_B}$.


\begin{proof}
From Lemma~\ref{lem:bound-L2} we know that
\[
\Delta(M_n,f)^2 \,=\, \E_u \E_v |I(f)-Q_{UB_n,v}(f)|^2 
\,=\, \E_u \sum_{k\in\Z^d\setminus\{0\}} |\F f(UB_nk)|^2.
\]
Using the monotone convergence theorem and $U=\diag(u)$ with 
$u\sim\U([1/2,3/2]^d)$ we obtain
\[
\Delta(M_n,f)^2 \,=\, \sum_{k\in\Z^d\setminus\{0\}} \int_{[1/2, 3/2]^d}|\F f(UB_nk)|^2\dint u.
\]
Now, for fixed $k$, we use the substitution 
$\xi=UB_nk=(u_1(B_nk)_1,\dots,u_d(B_nk)_d)$ and define the axis-parallel boxes 
$R_k:=\prod_{j=1}^d\left[\frac12(B_nk)_j, \frac32(B_nk)_j\right]$
to obtain 
\[\begin{split}
\Delta(M_n,f)^2 \,&=\, \sum_{k\in\Z^d\setminus\{0\}} 
	\int_{R_k}\frac{|\F f(\xi)|^2}{\prod_{j=1}^d|(B_nk)_j|}\dint \xi\\
\,&=\, \sum_{k\in\Z^d\setminus\{0\}} \int_{\R^d} \ind_{R_k}(\xi) \;
	\frac{|\F f(\xi)|^2}{\prod_{j=1}^d|(B_nk)_j|}\dint \xi\\
\,&=\, \int_{\R^d} |\F f(\xi)|^2 \sum_{k\in\Z^d\setminus\{0\}} 
	\frac{\ind_{R_k}(\xi)}{\prod_{j=1}^d|(B_nk)_j|}\dint \xi.
\end{split}\]
From Lemma~\ref{lem:boxes} we obtain  
\[\begin{split}
\sum_{k\in\Z^d\setminus\{0\}} \frac{\ind_{R_k}(\xi)}{\prod_{j=1}^d|(B_nk)_j|}
\,&=\, \sum_{k\in\Z^d\setminus\{0\}} \frac{\ind_{[\frac23 \xi,2\xi]}(B_nk)}{\prod_{j=1}^d|(B_nk)_j|}
\,\le\, \frac{(3/2)^d}{\prod_{j=1}^d|\xi_j|}\, \sum_{k\in\Z^d\setminus\{0\}} \ind_{[\frac23 \xi,2\xi]}(B_nk)\\
\,&\le\, \frac{(3/2)^d}{\prod_{j=1}^d|\xi_j|}\, \Bigl\vert B_n\bigl(\Z^d\setminus\{0\}\bigr)\cap [0,2\xi] \Bigr\vert\\ 
\,&\le\, \frac{3^d d_B}{n}\,\ind_{D_n}(\xi).
\end{split}\]
This 
proves the result.\\
\end{proof}

\begin{rem}\label{rem:dB}
The number $d_B$ is the determinant of the matrix $B$ that satisfies \eqref{eq:prop1}. 
Although we presently do not know how to find ``good'' matrices, we still 
want to know if there are matrices that make the involved constants small. 
Unfortunately, this is not the case. The quantity $D^*:=\inf_B d_B$, where the infimum is 
taken over all $B$ that satisfy \eqref{eq:prop1}, is a central object in the 
\emph{geometry of numbers}, see e.g.~\cite{Le69} for a comprehensive treatment of this 
topic.
There, $D^*$ is called the critical determinant of the star-body 
$S_d:=\{x\in\R^d\colon |x_1\cdot\dots\cdot x_d|\le1\}$ 
(denoted by $\Delta(S_d)$) 
and it is proven that 
$D^*\ge d^d/d!$, see~\cite[Section~41.2]{Le69}. 
Hence the upper bounds that are provided by Theorem~\ref{thm:error-L2} are in any case 
exponentially large in $d$. \\
It remains a challenging open problem if, 
for some $\a>1/2$ and $c_d$ is bounded by a polynomial in $d$, 
an error bound of the form $c_d\, n^{-\a}$
is even possible for, say, functions in ${\bf H}^s_p([0,1]^d)$ with large $s$. 
For $\a=1/2$ this is achieved by the classical Monte Carlo method for functions in 
$L_2([0,1]^d)$.
\end{rem}

\section{Error bounds for smooth functions} \label{sec:sobolev}

In this section we prove the error bounds of the randomized Frolov 
cubature rule for several classes of smooth functions. 
Here we still assume that the functions are defined 
on the whole $\R^d$ and have support inside a bounded, measurable set $\O$
with volume~1.

The function classes under consideration are \emph{Sobolev spaces of 
isotropic/anisotropic/mixed smoothness}. 
In the sequel, $\nu\colon\R^d\to\R$ is always a measurable function 
with $|\nu|>0$. 
Let $1<p<\infty$ and define the spaces
\begin{equation}\label{eq:space-nu}
H_p^\nu \,:=\, \left\{f\in L_p(\R^d)\colon\, 
\F^{-1}\left[\nu\cdot\F f\right]\in L_{p}(\R^d)\right\}
\end{equation}
and
\begin{equation}\label{eq:space-nu2}
\mr{H}_p^\nu(\O) \,:=\, \left\{f\in H_p^\nu\colon\, 
\supp(f)\subset\O\right\}
\end{equation}
equipped with the norm 
$\|f\|_{H_p^\nu} = \|\F^{-1}\left[\nu\cdot\F f\right]\|_{L_{p}(\R^d)}$.
For $S\in\R_+^d$, we denote the Sobolev spaces of anisotropic smoothness $S$ by
\begin{equation}\label{eq:norm1}
H_p^S  \qquad  \text{if} \qquad \nu(\xi)=\nu_S(\xi) := 1+\sum_{j=1}^d|2\pi \xi_j|^{S_j}
\end{equation}
and the Sobolev spaces of anisotropic mixed smoothness $S$ by
\begin{equation}\label{eq:norm2}
\quad\;\, {\bf H}_p^S \qquad \text{if}  \qquad  \nu(\xi)=\wt{\nu}_S(\xi) := \prod_{j=1}^d\left(1+\,|2\pi \xi_j|^{S_j}\right).
\end{equation}
In the case that $S_1=\ldots=S_d=s\in\R_+$ we replace $S$ by $s$ in the above 
notation and denote the spaces \emph{Sobolev spaces of isotropic (resp. mixed) 
smoothness $s$}.
It is well-known that for $S\in\N_0^d$ we can equivalently norm the spaces by
\begin{equation}\label{eq:norm3}
\|f\|_{H^S_p} \,=\, \|f\|_{L_p}+\sum_{j=1}^d \|D^{S_j\cdot e_j} f\|_{L_p}
\end{equation}
and
\begin{equation}\label{eq:norm4}
\|f\|_{{\bf H}^S_p} \,=\, \|f\|_{L_p}+ \sum_{\a\colon \alpha_j\in\{0, S_j\}}\|D^{\a} f\|_{L_p},
\end{equation} 
respectively, where $D^\a f$, $\a\in\N_0^d$, denotes the usual (weak) 
partial derivative of a function~$f$ and 
$e_j$ is the $j$th unit vector in $\R^d$. 

\begin{rem}\label{rem:def}
We use the norms and the corresponding spaces from \eqref{eq:norm3} and 
\eqref{eq:norm4} also for $p=1$ and $p=\infty$. 
Note that the definitions from \eqref{eq:space-nu}--\eqref{eq:norm2} make 
also sense for $p=1$, however in this case they are usually not called 
Sobolev spaces. 
Moreover, note that for $S\in\N^d$ the spaces above are the classical 
Sobolev spaces of (mixed) smoothness $S$, while for $S\notin\N^d$ these 
spaces are sometimes called \emph{Bessel potential spaces}. 
These spaces appear as complex interpolation spaces between Sobolev 
spaces of integer smoothness and are in the scale of 
Triebel-Lizorkin spaces.
For more details on these spaces as well as a historical treatment and 
further results see e.g.~\cite{DTU16, Te93, Tr83}.
But note that the spaces appear in these references also with other denotations, 
like $W^s_p$, ${\bf W}^s_p$ (see~\cite{DTU16}) 
or $S_p^s W$ (see~\cite{Te93}).
\end{rem}

\begin{rem}\label{rem:sob}
There are several different natural definitions of the norms for Sobolev spaces 
of the above type. In particular, one could replace the 
$\ell_{1}$-norms in \eqref{eq:norm1}--\eqref{eq:norm4} 
by any other $\ell_q$-norm, $1\le q\le\infty$, 
since all these norms are equivalent as long as $d$ is finite. 
This would only result in additional constants. 
There are also different conventions for the set of derivatives. 
For example, some people choose 
$\|f\|_{{\bf H}^S_p} = \sum_{\a\colon\a\le S}\|D^\a f\|_{L_p}$ 
instead of \eqref{eq:norm4}.
However, the corresponding spaces are equal.\\
\end{rem}

Before we proceed with the results for the 
Sobolev spaces as defined above, we state a result which will be the common 
starting point for the error bounds in the specific cases.
The following is a direct consequence of Theorem~\ref{thm:error-L2}'. 

\begin{prop} \label{prop:error-general}
Let $M_n$, $n>0$, be given by \eqref{eq:alg} and $f\in \mr{H}_2^\nu(\O)$. 
Moreover, we define the set 
$D_n=\{\xi\in\R^d \colon \prod_{j=1}^d|2 \xi_j|\ge n/d_B\}$, cf.~\eqref{eq:prop}. 
Then,
\[
\Delta(M_n,f) \,\le\, c_d\, n^{-1/2}\, 
	\|\nu^{-1}\|_{L_\infty(D_n)}\, \|f\|_{H^\nu_2}
\]
with $c_d=3^{d/2}\sqrt{d_B}$.
\end{prop}

\begin{proof}
In view of Theorem~\ref{thm:error-L2}' it is enough to prove the corresponding 
bound on the norm of $\F f$. We obtain from H\"older's inequality that
\[
\|\F f\|_{L_2(D_n)} \,=\, \|\nu^{-1}\cdot\nu\cdot\F f\|_{L_2(D_n)} 
\,\le\, \|\nu^{-1}\|_{L_\infty(D_n)} \cdot \|\nu\cdot\F f\|_{L_{2}(\R^d)}.
\]
Additionally, 
we obtain $\|\nu\cdot\F f\|_{L_{2}(\R^d)}=\|\F[\nu\cdot\F f]\|_{L_{2}(\R^d)}$ 
from the Plancharel theorem, since $\nu\cdot\F f\in L_2(\R^d)$ by assumption. 
This proves the result.\\
\end{proof}

\goodbreak


We see that for the proof of the error bounds for Sobolev spaces 
with $p=2$ it just 
remains to bound some $L_\infty$-norm of the function $1/\nu$. 
%
The proofs of these bounds are quite standard. 
However, we present them for convenience.

\begin{lem}\label{lem:norm1}
Let $\nu_S$, $S\in\R_+^d$, from \eqref{eq:norm1} and 
$D_n=\{\xi\in\R^d \colon \prod_{j=1}^d|2 \xi_j|\ge n/d_B\}$.
Additionally, define $g(S)=(\sum_{j=1}^d 1/S_j)^{-1}$ for $S>0$ 
and $g(S)=0$ otherwise. 
Then, we have
\[
\|\nu_S^{-1}\|_{L_\infty(D_n)} \,\lesssim\, n^{-g(S)}.
\]
The hidden constant only depends on $d$, $S$ and $B$.
\end{lem}

\begin{proof}
We clearly have $\nu_S(\xi)\ge1$. This already proves the result 
if $S_j=0$ for some $j$.
Now assume $S>0$ and define $\omega_j:=g(S)/S_j$, such that $\sum_{j=1}^d\omega_j=1$. 
From the weighted arithmetic-geometric mean inequality, we obtain
\[
\nu_S(\xi) \,\ge\,\sum_{j=1}^d \omega_j |2\pi \xi_j|^{S_j}
\,\ge\, \prod_{j=1}^d |2\pi \xi_j|^{\omega_j S_j}
\,=\, \left(\prod_{j=1}^d |2\pi \xi_j|\right)^{g(S)}.
\]
This implies $\|\nu^{-1}\|_{L_\infty(D_n)}\lesssim n^{-g(S)}$
and proves the statement. \\
\end{proof}

\begin{lem}\label{lem:norm2}
Let $\wt\nu_S$, $S\in\R_+^d$, from \eqref{eq:norm2} and 
$D_n=\{\xi\in\R^d \colon \prod_{j=1}^d|2 \xi_j|\ge n/d_B\}$.
Then, we have
\[
\|{\wt\nu_S}^{-1}\|_{L_\infty(D_n)} \,\lesssim\, n^{-s_{\min}},
\]
where $s_{\min}=\min_j S_j$. 
The hidden constant only depends on $d$, $S$ and $B$.
\end{lem}

\begin{proof}
We have 
\[
\wt\nu_S(\xi) \,\ge\, \prod_{j=1}^d\max\{1,|2\pi\xi_j|\}^{S_j}
\,\ge\, \left(\prod_{j=1}^d|2\pi\xi_j|\right)^{S_{\min}}.
\]
This proves the statement. \\
\end{proof}

\medskip

For $p>2$ we just use the embedding 
$\mr{H}^\nu_p(\O)\hookrightarrow \mr{H}^\nu_2(\O)$, 
see~\eqref{eq:space-nu2}, which follows from the compact support of the 
contained functions, see e.g.~\cite[Thm.~3.3.1(iii)]{Tr83}. 
That is, we use for $p>2$ the inequalities
\[
\|f\|_{H^S_2} \,\lesssim\, \|f\|_{H^S_p} \qquad \text{ for }\; f\in \mr{H}^S_p(\O)
\]
and
\[
\|f\|_{{\bf H}^S_2} \,\lesssim\, \|f\|_{{\bf H}^S_p} \qquad \text{ for }\; f\in \mr{\bf H}^S_p(\O).
\]

The case $1\le p<2$ is a bit more involved. 
In the isotropic case we use the embedding
\[
\mr{H}^S_p(\O)\,\hookrightarrow\, \mr{H}^{S'}_2(\O)
\]
where $S'=\kappa\cdot S$ (component-wise) with 
$\kappa=1-g(S)^{-1}(1/p-1/2)$ if $g(S)\ge1/p-1/2$, 
see~\cite[Theorem~7]{JS07} and \cite{Tr83}. 
Using Proposition~\ref{prop:error-general} and Lemma~\ref{lem:norm1} 
we obtain
\[
\Delta(M_n,f) \,\lesssim\, n^{-g(S')-1/2}\, \|f\|_{H^{S'}_2} 
\,\lesssim\, n^{-g(S')-1/2}\, \|f\|_{H^{S}_p}
\]
for $f\in \mr{H}^S_p(\O)$, if $g(S)\ge1/p-1/2$. Finally, note that $g(S')=\kappa g(S)=g(S)-1/p+1/2$.
For spaces of mixed smoothness we use the chain of embeddings
\[
{\bf H}^S_p \,\hookrightarrow\, {\bf H}^{s_{\min}}_p 
\,\hookrightarrow\, {\bf H}^{s_{\min}-1/p+1/2}_2
\]
for $1\le p <2$ and 
$s_{\min}=\min_j S_j$ with 
$s_{\min}\ge 1/p-1/2$, see e.g.~\cite[Chapter~2]{ST87}. 
We obtain with Proposition~\ref{prop:error-general} and Lemma~\ref{lem:norm2} 
that
\[
\Delta(M_n,f) \,\lesssim\, n^{-s_{\min}-1+1/p}\, \|f\|_{{\bf H}^{s_{\min}-1/p+1/2}_2} 
\,\lesssim\, n^{-s_{\min}-1+1/p}\, \|f\|_{{\bf H}^{S}_p}
\]
for $f\in \mr{\bf H}^S_p(\O)$.

We now summarize the results of this section.

\begin{thm} \label{thm:main1}
Let $M_n$, $n>0$, be given by \eqref{eq:alg}, $S\in\R_+^d$ and 
$1\le p\le \infty$ ($p\neq1,\infty$ if $S\notin\N^d$). 
Then, for $f\in \mr{H}^S_p(\O)$, 
\[
\Delta(M_n,f) \,\lesssim\, n^{-g(S)-\min\{1/2, 1-1/p\}}\; \|f\|_{H^S_p}, 
\]
if $g(S)\ge\max\{0, 1/p-1/2\}$, where $g(S)=(\sum_{j=1}^d 1/S_j)^{-1}$.
The hidden constant only depends on $p$, $d$, $S$ and $B$.
Moreover, $N(M_n)=n$.
\end{thm}

\medskip

\begin{thm} \label{thm:main2}
Let $M_n$, $n>0$, be given by \eqref{eq:alg}, $S\in\R_+^d$ and 
$1\le p\le \infty$ ($p\neq1,\infty$ if $S\notin\N^d$). 
Then, for $f\in \mr{\bf H}^S_p(\O)$, 
\[
\Delta(M_n,f) \,\lesssim\, n^{-s_{\min}-\min\{1/2, 1-1/p\}}\, \|f\|_{{\bf H}^S_p}, 
\]
if $s_{\min}\ge\max\{0, 1/p-1/2\}$, where $s_{\min}=\min_j S_j$.
The hidden constant only depends on $p$, $d$, $S$ and $B$. 
Moreover, $N(M_n)=n$.
\end{thm}

\medskip
\section{Integration of functions on the cube} \label{sec:non-periodic}

Until now we always considered functions that are supported inside 
a bounded set $\Omega$ of volume one. This was for two reasons. 
First of all, this was necessary to ensure that the algorithm $M_n$ from 
\eqref{eq:alg} uses in expectation exactly $n$ function evaluations. 
Additionally, it was necessary for the results in 
Theorems~\ref{thm:main1}~\&~\ref{thm:main2} for $p>2$, since the used 
embeddings only work for functions defined on bounded sets.

In this section we comment on the integration of functions that are 
defined on the unit cube $\Omega=[0,1]^d$ and do not satisfy any boundary 
condition.
These spaces are defined as restriction of the spaces $H^\nu_p$, 
see~\eqref{eq:space-nu}, to $[0,1]^d$. That is we define
\begin{equation}\label{eq:space-nu-cube}
H_p^\nu([0,1]^d) \,:=\, \left\{f\in L_p([0,1]^d)\colon\, 
\exists g\in H^\nu_p\; \text{ such that }\; g|_{[0,1]^d}=f\right\}
\end{equation}
with the (quasi-)norm 
\[
\|f\|_{H_p^\nu([0,1]^d)} \,:=\, \inf_{g}\, \|g\|_{H_p^\nu},
\]
where the infimum is taken over all functions $g\in H^\nu_p$ that 
agree with $f$ on $[0,1]^d$.
Again we consider the choices of $\nu$ and the notation from 
\eqref{eq:norm1} and \eqref{eq:norm2}
and denote the corresponding spaces by $H_p^S([0,1]^d)$ and 
${\bf H}_p^S([0,1]^d)$, respectively. 

The algorithm that is used for these spaces is based on the algorithm 
$M_n$ from \eqref{eq:alg} together with a mapping 
$T$ that maps boundedly from $H_p^\nu([0,1]^d)$ to $\mr{H}_p^\nu$. 
Such mappings and their application to numerical integration appeared 
several times in the literature, see e.g.~\cite{By85,Du93,Du97,NUU16,Te93,Te03}. 
Here, we follow~\cite{Te03} and use componentwise change of variable with 
a suitable $C^\infty(\R)$-function~$\psi$, i.e.
\begin{equation}\label{eq:psi}
\psi(t) := \left\{\begin{array}{rcl}
          \int_0^t e^{-\frac{1}{\xi(1-\xi)}}\,d\xi / \int_0^1 e^{-\frac{1}{\xi(1-\xi)}} \,d\xi &:& t\in [0,1],\\
          1&:& t>1,\\
          0&:& t<0\,.
\end{array}\right.
\end{equation}
We define
\[
Tf(x) \,:=\, \left|\prod_{j=1}^d\psi'(x_j)\right| \, f\bigl(\psi(x_1),\ldots,\psi(x_d)\bigr), 
\qquad x\in\R^d.
\]
Clearly, $\supp(Tf)\subset[0,1]^d$ and, by 
change of variable, 
$\int_{[0,1]^d}Tf(x)\dint x=\int_{[0,1]^d}f(x)\dint x$.

For functions $f\in H_p^\nu([0,1]^d)$ we consider the randomized algorithm 
\begin{equation}\label{eq:alg2}
\widebar{M}_n(f) \,:=\, M_n(Tf),
\end{equation}
where $M_n$ is given in \eqref{eq:alg}.
From the results of the previous sections, 
see e.g.~Proposition~\ref{prop:error-general}, we know that we can bound the 
mean squared error of $\widebar{M}_n$ by 
\[
\Delta(\widebar{M}_n,f) \,=\, \Delta(M_n,Tf) \,\le\, e_n(\nu,p,d)\cdot \|Tf\|_{H^\nu_p}
\]
for some $e_n(\nu,p,d)$ that is independent of $f$. 
To prove the desired error bounds it remains to show 
$\|Tf\|_{H^\nu_p}\lesssim\|f\|_{H^\nu_p([0,1]^d)}$, i.e.~that 
$T: H^\nu_p([0,1]^d)\to \mr{H}^\nu_p$ is bounded. 
If so, this shows that we have the same (up to a constant) error bound for 
$\widebar{M}_n$ in $H^\nu_p([0,1]^d)$ as we have for $M_n$ in $\mr{H}^\nu_p$.

For the spaces $H^S_p$ and ${\bf H}^S_p$, $S\in\R_+^d$, $1\le p \le\infty$ 
($1<p<\infty$ if $S\notin\N^d$), this boundedness was shown in 
\cite{Te03} and \cite{NUU16}. Actually, the boundedness was only proven 
for the cases $S_1=\ldots=S_d$, but the proofs in the anisotropic case 
follow exactly the same lines.
For a more detailed treatment of such ``change of variable''-mappings 
(especially for the use of piecewise polynomials instead of $\psi$) 
see~\cite{NUU16} and the references therein.

\bigskip

\noindent
{\bf Acknowledgement.}
The author thanks Stefan Heinrich, Aicke Hinrichs, David Krieg, 
Erich Novak and Tino Ullrich for many fruitful discussions on the 
subject of this paper. 
Additionally, I thank Glenn Byrenheid and Tino Ullrich for hints 
that lead to substantial improvements in Section~\ref{sec:sobolev}, 
and Andreas M\"uller for his valuable and inspiring comments.

\end{document}